\documentclass[letterpaper, 10 pt, conference]{ieeeconf}  

\usepackage{graphics} 
\usepackage{epsfig} 
\usepackage{mathptmx} 
\usepackage{times} 
\usepackage{amsmath} 
\usepackage{amssymb}  
\usepackage{amssymb,theorem}
\usepackage{mathrsfs}
\usepackage{hyperref}
\usepackage{color,xspace}
\usepackage{psfrag}
\usepackage{graphicx}
\usepackage{subcaption}
\usepackage{cite}
\usepackage{cleveref}

\newtheorem{theorem}{Theorem}
\newtheorem{lemma}{Lemma}	

\newtheorem{remark}{Remark}

\newtheorem{assumption}{Assumption}

\def\bx{\textbf{x}}
\def\bz{\textbf{z}}
\def\by{\textbf{y}}

\title{\LARGE \bf Interior Point Method for Dynamic Constrained Optimization in Continuous Time}
\author{Mahyar Fazlyab, Santiago Paternain, Victor M. Preciado and Alejandro Ribeiro 
	\thanks{The authors are with the Department of Electrical and Systems Engineering, University of Pennsylvania, Philadelphia, PA 19104, USA. Email: \{mahyarfa, spater, preciado, aribeiro\}@seas.upenn.edu.}}

\begin{document}

\maketitle
\thispagestyle{empty}
\pagestyle{empty}

\begin{abstract}
	This paper considers a class of convex optimization problems where both, the objective function and the constraints, have a continuously varying dependence on time. Our goal is to develop an algorithm to track the optimal solution as it continuously changes over time inside or on the boundary of the dynamic feasible set. We develop an interior point method that asymptotically succeeds in tracking this optimal point in nonstationary settings. The method utilizes a time varying constraint slack and a prediction-correction structure that relies on time derivatives of functions and constraints and Newton steps in the spatial domain. Error free tracking is guaranteed under customary assumptions on the optimization problems and time differentiability of objective and constraints. The effectiveness of the method is illustrated in a problem that involves multiple agents tracking multiple targets.
\end{abstract}


\section{Introduction}
In a conventional optimization problem we are given a fixed objective and a fixed constraint and are tasked with finding the optimal argument that minimizes the objective among all feasible variables. In a time varying problem the objective and constraints change continuously in time and we are tasked with tracking the optimal point as it varies over time. These problems arise often in dynamical systems and control because many practical situations involve an objective function or a set of constraints that have dependence on time \cite{aastrom2013adaptive,nelles2001nonlinear,fazlyab2013parameter, hofleitner2011online}. Particular examples include estimation of the path of a stochastic process \cite{jakubiec2013d}, signal detection with adaptive filters \cite{cavalcante2013distributed}, tracking moving targets\cite{zhou2011multirobot}, and scheduling trajectories in an autonomous team of robots. \cite{tu2011mobile}

Methods to solve convex optimization problems -- say, gradient descent, Newton's method, and interior point -- are iterative in nature \cite{boyd2004convex,fiacco1990nonlinear}. When applied to a time varying nonstationary setting, each iteration moves the argument closer to the optimum while the optimum drifts away because of the changing nature of the objective and the constraints. This process is likely to settle into a steady state optimality gap that depends on the relative time constants of the dynamical process and the optimization algorithm. That this is indeed true has been observed and proven for gradient descent in unconstrained optimization \cite{Popkov2005}, as well as in constrained optimization problems that arise in the specific contexts of distributed robotics \cite{zavlanos2013network}, sequential estimation \cite{jakubiec2013d} and distributed optimization with a time varying alternating direction method of multipliers \cite{ling2013decentralized}.  

Alternatively, one can draw inspiration from the prediction-correction structure of Bayesian filters and utilize knowledge of the system's dynamics to predict the drift of the optimal operating point and utilize the descent step of an optimization algorithm to correct the prediction. Variations of this idea have been developed in discrete \cite{simonetto2015class} and continuous \cite{baumann2004newton} time. When used in discrete time, the addition of a prediction step has been shown to reduce the tracking error relative to verbatim use of a descent algorithm \cite{simonetto2015class, simonetto2015prediction}. When used in continuous time, the use of a prediction step and a Newton correction results in perfect tracking of the optimal argument of an \textit{ unconstrained} optimization problem \cite{baumann2004newton,rahili2015distributed}. 

This paper develops an interior point method to track the optimal point of a convex time varying {\it constrained} optimization problem (Section II). Important characteristics of this method are: (i) The use of a time time varying logarithmic barrier akin to the barrier used in static interior point methods. (ii) The use of a time varying constraint slack that is decreased over time and guarantees asymptotic satisfaction of the constraints. (iii) The use of time derivatives that play the role of a prediction step that tries to follow the movement of the optimal argument. (iv) The use of spatial Newton decrements that play the role of a correction step by pushing towards the current optimum. The main contribution of this paper is to show that this method converges to the time varying optimum under mild assumptions (Section III). These assumptions correspond to the customary requirements to prove convergence of interior point methods and differentiability of the objective and constraints with respect to time variations (Theorem 1). It is important to emphasize that our convergence result holds for nonstationary systems and, as such, do not rely on a vanishing rate of change. This implies that the proposed systems succeeds in tracking the optimum without error after a transient phase. The effectiveness of the method is illustrated in a problem that involves multiple agents tracking multiple targets (Section IV).

\textit{Notation and Preliminaries.} Given an n-tuple $(x_1, . . . , x_n)$, $\mathbf{x} \in \mathbb{R}^n$ is the associated vector. We denote as $\mathbf{I}_n$ the n-dimensional identity matrix, as $\mathbb{S}^{n}$ the space of symmetric matrices and as $\mathbb{S}^{n}_{++}$ and $\mathbb{S}^{n}_{+}$ the spaces of positive definite and positive semidefinite matrices, respectively. For square matrices $\mathbf{A}$ and $\mathbf{B}$, we write $\mathbf{A} \succeq \mathbf{B}$ if and only if $\mathbf{A}-\mathbf{B}$ is positive semidefinite. The Euclidean norm of a vector $\mathbf{x}$ is $\|\mathbf{x}\|_2$. The gradient of the function $f(\bx,t) \in \mathbb{R}$ with respect to $\bx \in \mathbb{R}^n$ is denoted by $\nabla_{\bx} f(\bx,t) \in 
\mathbb{R}^n$. The partial derivatives of $\nabla_{\bx} f(\bx,t)$with respect to $\bx$ and $t$ are denoted by $\nabla_{\bx\bx} f(\bx,t) \in \mathbb{S}^n$ and $\nabla_{\bx t} f(\bx,t) \in \mathbb{R}^n$, respectively.


\section{Problem statement}\label{se:statement}
Consider the following constrained convex optimization problem
\begin{align}\label{test}
\bx^{\star}:=& \arg \min_{\bx\in\mathbb{R}^n}  & & f_{0}(\bx) \\ 
&\mbox{s.t.} & & f_{i}(\bx) \leq 0, \quad i\in \{1,\cdots,p\}. \nonumber
\end{align} 
In order to solve \eqref{test}, we can exploit interior point method \cite{fiacco1990nonlinear,potra2000interior} in which we relax the constraints and penalize their violation by logarithmic functions of the form $ -\log\left(-f_i(\bx)\right)$. More specifically, we solve the relaxed problem
\begin{equation}\label{eqn_relaxed_time_invarant}
\bx^{\star}(c):=\arg\min_{\bx\in\mathcal{D}} \, \Phi(\bx,c),
\end{equation}
where the so-called barrier function $\Phi(\bx,c)$ is defined as
\begin{align} \label{eq: time_invariant_barrier}
\Phi(\bx,c)= f_0(\bx) - \frac{1}{c}\sum_{i=1}^p \log(-f_i(\bx)),
\end{align}
and $\mathcal{D}=\{\bx\in \mathbb{R}^n| f_i(\bx)<0, i=1,\cdots,p\}$ is the interior of the feasible domain. Furthermore, $c$ is a positive constant such that $\bx^{\star}(c) \to \bx^{\star}$ as $c \to \infty$. 

To implement the interior point method, the unconstrained optimization problem \eqref{eqn_relaxed_time_invarant} is solved sequentially for a positive growing sequence $\{c_k\}$, each starting from the optimal solution of the previous optimization problem. The resulting sequence $\{\bx^{\star}(c_k)\}$ converges to the optimal point as $c_k \to \infty$. For each fixed $c_k$, $\bx^{\star}(c_k)$ can be found, for instance, by Newton's method as follows
\begin{align} \label{eq: time_invariant_newton_barrier}
\dfrac{d}{dt} \bx(t) = -\nabla_{\bx\bx}^{-1} \Phi(\bx(t),c_{k})\nabla_{\bx}\Phi(\bx(t),c_{k}).
\end{align}
with initial condition $\bx(0)=\bx^{\star}(c_{k-1})$. In order to decrease the total convergence time to the optimal point $\bx^{\star}(t)$, it is appealing to increase $c$ as a function of time, in lieu of discontinuous updates. Is this case, problem \eqref{eq: time_invariant_barrier} would involve a time varying objective function.

In this paper, we consider a more general case in which both the objective function and/or the constraints are time-varying. More formally, we consider the following problem
%
\begin{align} \label{eq: inequality_constrained_time_varying_problem}
\bx^{\star}(t):=& \arg \min_{\bx\in\mathbb{R}^n}  & & f_{0}(\bx,t) \\ 
&\mbox{s.t.} & & f_{i}(\bx,t) \leq 0, \quad i\in \{1,\cdots,p\}. \nonumber
\end{align} 
for all $t \in [0,\infty)$. The ultimate goal is generate a (not necessarily feasible) solution {$\bx(t)$} such that {$\bx(t) \to \bx^{\star}(t)$} as $t \to \infty$, which necessarily enforces asymptotic feasibility. The corresponding time-varying barrier function of \eqref{eq: inequality_constrained_time_varying_problem} becomes
\begin{equation} \label{eq: time_varying_barrier_function}
\Phi(\bx,t) = f_0(\bx,t) - \frac{1}{c(t)}\sum_{i=1}^p \log\left(-f_i(\bx,t)\right), \ \bx\in \mathcal{D}(t)
\end{equation}
where $\mathcal{D}(t):=\{\bx \in \mathbb{R}^n | f_i(\bx,t)<0,\ i=1,\cdots,p\}$ is the interior of the (time varying) feasible region, and $c(t)$ is a positive valued function of time. For minimizing \eqref{eq: time_varying_barrier_function}, we will propose a time-varying Newton differential equation whose solution $\bx(t)$ converges to $\bx^{\star}(t)$ as $t \to \infty$.

For further analysis, we assume that the objective function is strongly convex in $\bx$ and that the constraints are convex in $\bx$ for all times $t\in[0,\infty)$. In addition, we assume that $f_i(\bx,t)$ is continuously differentiable with respect to time for all $i \in \{0,1,\cdots,p\}$. We formalize these assumptions next. 
\begin{assumption} \label{assumption: strong convexity}
	The objective function $f_0(\bx,t)$ and the constraint functions $f_i(\bx,t)$ are twice continuously differentiable with respect to $\bx$ and continuously differentiable with respect to time for all $t\geq 0$. Furthermore, $f_0(\bx,t)$ is uniformly strongly convex in $\bx$, i.e., $\nabla_{\bx\bx} f_0(\bx,t) \succeq m \mathbf{I}$ for some $m>0$ and $f(\bx,t)$ is convex with respect to $\bx$ for all $t\geq 0$.
\end{assumption}
\begin{assumption} \label{assumption: strict_feasibility}
	Slater's condition qualification holds for problem \eqref{eq: inequality_constrained_time_varying_problem} for all $t \geq 0$, i.e., there exits $\bx^\dagger \in \mathbb{R}^n$ such that $f_i(\bx^\dagger,t)<0$ for all $t\geq 0$. 
\end{assumption}

The above assumptions make the convexity-related properties to be invariant over time. With Assumption \ref{assumption: strict_feasibility}, the necessary and sufficient condition for optimality of problem \eqref{eq: inequality_constrained_time_varying_problem} at all times $t\geq 0$ read as
\begin{align} \label{eq: time_varying_KKT_conditions}
\nabla_{\bx} f_0(\bx^{\star}(t),t)+\sum_{i=1}^{p}\lambda_{i}^{\star}(t) \nabla_{\bx} f_{i}(\bx^{\star}(t),t)&=0, \\
\lambda_{i}^{\star}(t) f_{i}(\bx^{\star}(t),t)&=0,\ i \in \{1,\cdots,p\}\nonumber \\
\lambda_{i}^{\star}(t) &\geq 0,\ i \in \{1,\cdots,p\} \nonumber \\
f_{i}(\bx^{\star}(t),t) &\leq 0 \nonumber.
\end{align}
where $\lambda^{\star}(t)=[\lambda_1^{\star}(t),\cdots,\lambda_p^{\star}]^T \in \mathbb{R}^p$ is the vector of optimal dual variables. Finally, we make a further assumption about the time variations of the optimal primal-dual pair.
\begin{assumption} \label{assumption: bounded_in_times}
	For any $\alpha>0$, the optimal dual variables satisfy $\lambda_i^{\star}(t)\exp(-\alpha t)\to 0$ as $t \to \infty$ for all $i \in \{1,\cdots,p\}$.
\end{assumption}

The above assumption excludes the possibility for the optimal dual variables (and hence the optimal primal variables) to escape to infinity exponentially fast. Otherwise, the optimality conditions would become ill-conditioned as $t \to \infty$ and its solution is not tractable in the implementation phase. Equivalent assumtions to  Assumption \ref{assumption: bounded_in_times} are made in similar settings. See e.g. \cite{simonetto2015class}.

Prior to solving the general problem \eqref{eq: inequality_constrained_time_varying_problem}, we start off with unconstrained dynamic convex optimization where the optimization space is the domain of the objective function (Section \ref{sec_unconstrained}). In Section \ref{sec_equality_constraint} we take time varying linear equality constraints into account and in Section \ref{sec_interior_point} we deal with the case of generic dynamic convex constrained optimization problems. In section \ref{sec_numerical_examples} two tracking problems are studied that are posed as time varying optimization problems that can be solved through the techniques developed in Section \ref{sec_main_section}.

\section{Time-Varying Interior Point Method}\label{sec_main_section}
In this section, we develop a time-varying interior point method that solves \eqref{eq: inequality_constrained_time_varying_problem}.  We first introduce time-varying Newton method for unconstrained dynamic optimization problems. Next, we generalize the method to linear equality constraints. Finally, we incorporate time-varying interior point method for the case of inequality constraints.
\subsection{Unconstrained Time-Varying Convex optimization}\label{sec_unconstrained}
In unconstrained time-varying convex optimization the goal is to track, continuously in time, the minimizer of a time-varying convex function. Mathematically speaking, given an objective function $f_0(\bx,t): \mathbb{R}^n \times \mathbb{R}_{+} \to \mathbb{R}$, the goal is estimate the trajectory $\bx^{\star}(t)$ where
\begin{align}  \label{eq: unconstrained_time_varying_problem}
\bx^{\star}(t):=\arg \underset{\bx \in \mathbb{R}^n}{\min} \quad f_{0}(\bx,t).
\end{align} 
The optimal trajectory {$\bx^{\star}(t)$} is characterized by the points where the gradient of {$f_0(\bx,t)$} with respect to $\bx$ is zero, i.e. $\nabla_{\bx} f_0(\bx^{\star}(t),t)\equiv 0$ for all $t \in [0,\infty)$. Using chain rule to differentiate the latter identity with respect to time yields
\begin{align}
 \dfrac{d}{dt} \nabla_{\bx} f_0(\bx^{\star}(t),t)= &\nabla_{\bx\bx} f_0(\bx^{\star}(t),t) \dfrac{d}{dt}\bx^{\star}(t)+\nabla_{\bx t} f_0(\bx^{\star}(t),t).
\end{align}
%
The left hand side of the above equation is identically zero for $t \in [0,\infty)$. It follows that the optimal solution moves with a velocity given by
\begin{align}
\dfrac{d}{dt}\bx^{\star}(t) =-\nabla_{\bx\bx}^{-1} f_0(\bx^{\star}(t),t)\nabla_{\bx t} f_0(\bx^{\star}(t),t).
\end{align}
The above observation suggests that the tracking trajectory should evolve with (approximately) the same velocity as the minimizer trajectory, while taking a descent direction at the same time in order to get closer to the optimal trajectory. If Newton-method is chosen as the descent direction, the resulting time-varying Newton method takes the form 
\begin{align} \label{eq: time_varying_newton}
	\dot{\bx}(t)=-{\nabla_{\bx\bx}^{-1} f_0(\bx(t),t)}[\mathbf{P}\nabla_{\bx} f_0(\bx(t),t)+\nabla_{\bx t} f_0(\bx(t),t)],
\end{align}
where $\mathbf{P}$ is a positive definite matrix. The next lemma shows that the solution of the dynamical system \eqref{eq: time_varying_newton} converges exponentially to the solution to the unconstrained minimization problem \eqref{eq: unconstrained_time_varying_problem}.
\begin{lemma} \label{lem: Time_Varying_Newton}
	Let $\bx^{\star}(t)$ be defined as in \eqref{eq: unconstrained_time_varying_problem} and $\bx(t)$ be the solution of the differential equation \eqref{eq: time_varying_newton}
	with $\mathbf{P} \in \mathbb{S}^{n}_{++}$ satisfying $\mathbf{P}\succeq\sigma{\mathbf{I}}$. Then, the following inequality holds
	\[
	\|\bx(t)-\bx^{\star}(t)\|_2 \leq C(\bx_0,m)e^{-\sigma t}, 
	\]
	where $\bx_0$ is an arbitrary initial point and $0\leq C(\bx_0,m)<\infty$.
	\begin{proof}
		The proof is found in Appendix subsection \ref{ap: time_varying_unconstrained_newton}.
	\end{proof}
\end{lemma}
The previous lemma confirms that the trajectory generated by \eqref{eq: time_varying_newton} converges exponentially to the optimal trajectory and therefore it is possible to solve the unconstrained time-varying optimization problem \eqref{eq: unconstrained_time_varying_problem} by discretization of the dynamical system \eqref{eq: time_varying_newton}. Next, we show that the same dynamical system allows us to solve dynamic optimization problems with equality constraints by augmenting the state space with the Lagrange multipliers associated with the equality constraints. 

\subsection{Equality-Constraint Time-Varying Convex Optimization}\label{sec_equality_constraint}
Consider the problem of tracking the minimizer of a convex time-varying objective function subject to linear constraints. More precisely, we seek to estimate $\bx^{\star}(t)$ for all $t\in [0,\infty)$ where
{\begin{align}  \label{eq: equality_constraint_time_varying_problem}
\bx^{\star}(t):=&\arg \underset{\bx \in \mathbb{R}^n}{\min} & & f_{0}(\bx,t) \\
&\mbox{s.t.} & & \mathbf{A}(t)\bx=\textbf{b}(t), \nonumber
\end{align} 
with $\mathbf{A}(t) \in \mathbb{R}^{p\times n}$ and $\mbox{rank}(\mathbf{A}(t))=p<n$ for all $t \geq 0$. The Lagrangian relaxation of \eqref{eq: equality_constraint_time_varying_problem} is 
\begin{align} \label{eq: time_varying_equality_lagrangian}
\mathcal{L}(\bx,\lambda,t)=f_0(\bx,t)+\lambda^T(\mathbf{A}(t)\bx-\textbf{b}(t)).
\end{align}
where $\lambda \in \mathbb{R}^p$ is the vector of Lagrange multipliers. The optimal trajectory $(\bx^{\star}(t),\lambda^{\star}(t))$ constitutes of points where the gradient of the Lagrangian with respect to both $\bx$ and $\lambda$ vanishes. Therefore, by extending the state space to be $\bz :=[ \bx^T \ \lambda^T]^T \in \mathbb{R}^{n+p}$, the optimal solution $\bz^{\star}(t)$ is characterized by
\begin{equation}
\nabla_{\bz} \mathcal{L} (\bz^{\star}(t),t) = 0, \ t \geq 0.
\end{equation}
Similar to the unconstrained case, the tracking trajectory should compose of two directions: a velocity compensating direction and a descent direction. The next lemma addresses the algorithm for tracking $\bx^{\star}(t)$ in \eqref{eq: equality_constraint_time_varying_problem}.

\begin{lemma} \label{lemma: equality_constrained_time_varyging_newton}
Denote $\bz(t)=[\bx(t)^T \ \lambda(t)^T]^T$ as the solution of the following differential equation
\begin{align} \label{eq: time_varyin_equality_newton_method}
\dfrac{d}{dt} \bz(t)= -\nabla_{\bz\bz}^{-1} \mathcal{L}(\bz(t),t)(\mathbf{P} \nabla_{\bz}\mathcal{L}(\bz(t),t)+\nabla_{\bz t}\mathcal{L}(\bz(t),t))
\end{align}
where $\mathcal{L}$ is defined in \eqref{eq: time_varying_equality_lagrangian} and $\mathbf{P}\in \mathbb{S}^{n+p}$ is a positive definite matrix satisfying $\mathbf{P} \succeq \sigma \mathbf{I}_{n+p}$. Then, the following inequality holds
\begin{align} \label{eq: time_varying_equality_constraint_convergence_bound}
\|\bx(t)-\bx^{\star}(t)\|_2^{2}+\|\lambda(t)-\lambda^{\star}(t)\|_2^{2} \leq C(\bx_0,\lambda_0,m) e^{-2\sigma t}.
\end{align}
where $\bx_0 \in \mathbb{R}^n$ and $\lambda_0 \in \mathbb{R}^p$ are arbitrary initial points and $0 \leq C(\bx_0,\lambda_0,m)<\infty$.
\end{lemma}

\begin{proof}
	See Appendix subsection \ref{ap: time_varying_equality_newton}.
\end{proof}

Notice that the dynamical system \eqref{eq: time_varyin_equality_newton_method} needs not to start from a feasible point, i.e. it is not required that $\mathbf{A}_0\bx_0=\mathbf{b}_0$. However, feasibility is achieved exponentially fast by \eqref{eq: time_varying_equality_constraint_convergence_bound}. 

In the sequel, we consider the most general case with time-varying inequality constraints.  Without loss of generally, we omit linear equality constraints as each single affine equality constraint can be expressed as two convex inequality constraints.
\subsection{Time-Varying Interior-point method}\label{sec_interior_point}
In this section, we return to the general optimization problem \eqref{eq: inequality_constrained_time_varying_problem} with the associated barrier function \eqref{eq: time_varying_barrier_function}. We will show that the same differential equation developed in Section \ref{sec_unconstrained} -- using the barrier function $\Phi(\bx,t)$ defined in \eqref{eq: time_varying_barrier_function} in lieu of $f_0(\bx,t)$ in \eqref{eq: time_varying_newton} -- pushes the generated solution to the optimal trajectory when the barrier parameter $c(t) \to \infty$. Formally, the dynamical system of interest is 
\begin{align} \label{eq: time_varying_newton_barrier}
\dot{\bx}(t)=-\nabla_{\bx\bx}^{-1}\Phi(\bx(t),t) \left( \mathbf{P}\nabla_{\bx}\Phi(\bx(t),t)+\nabla_{\bx t}\Phi(\bx(t),t) \right),
\end{align}
where $\mathbf{P} \in \mathbb{S}^n_{++}$. Notice, however, that the Newton method in this case needs to start from a strictly feasible point, i.e. $\bx_0 \in \mathcal{D}_0$. This limitation is not desirable, as it is not always straightforward to find such an initial condition. This restriction can be overcome by expanding the feasible region at $t=0$ by a slack variable, denoted by $s$, and shrink it to the real feasible set over time. More precisely, we perturb problem \eqref{eq: inequality_constrained_time_varying_problem} at each time $t \in [0,\infty)$ by $s(t): \mathbb{R}_{+} \to \mathbb{R}_{++}$ as follows
\begin{align} \label{eq: perturbed_inequality_constrained_time_varying_problem}
\tilde{\bx}^{\star}(t):=& \arg \min_{\bx\in\mathbb{R}^n}  & & f_{0}(\bx,t) \\ 
&\mbox{s.t.} & & f_{i}(\bx,t) \leq s(t), \quad i\in \{1,\cdots,p\}. \nonumber
\end{align} 
It can be observed that the feasible region is enlarged for $s(t)\geq0$. In particular, at time $t=0$, any initial point $\bx_0 \in \mathbb{R}^n$ can be made feasible by choosing the initial $s_0=s(0)$ large enough. Another consequence of such perturbation is that, the optimal value of the perturbed problem \eqref{eq: perturbed_inequality_constrained_time_varying_problem} is no larger than the optimal value at $\bx^{\star}(t)$. The next lemma formalizes this observation.
\begin{lemma} \label{lemma: pertubation_suboptimality_bounds}
	Let $\bx^{\star}(t)$ be defined as in \eqref{eq: inequality_constrained_time_varying_problem} and $\tilde{\bx}^{\star}(t)$ as in \eqref{eq: perturbed_inequality_constrained_time_varying_problem}. Then, the following inequality holds:
	\begin{align} \label{eq: pertubation_suboptimality_bounds}
	0 \leq f_0(\bx^{\star}(t),t)-f_0(\tilde{\bx}^{\star}(t),t) \leq \sum_{i=1}^{p} \lambda^{\star}_i(t) s(t)
	\end{align}
	where $\lambda_i^{\star}(t),\ i\in \{1,\cdots,p\}$ are optimal dual variables defined in \eqref{eq: time_varying_KKT_conditions}.
	
	\begin{proof}
		See Appendix subsection \ref{ap_perturbation_suboptimality_bound}.
	\end{proof}	
\end{lemma}
The above lemma asserts that the sub-optimality of the perturbed solution $\tilde{\bx}^{\star}(t)$ is controlled by $s(t)$. More importantly, as a result of Assumption \ref{assumption: bounded_in_times}, the sub-optimality can be pushed to zero if $s(t)=s_0 \exp(-\alpha t)$ for any $\alpha>0$. 

The barrier function associated with the problem \eqref{eq: perturbed_inequality_constrained_time_varying_problem} is
\begin{align} \label{eq: perturbed_time_varying_barrier_function}
\tilde{\Phi}(\bx,t) = f_0(\bx,t) - \frac{1}{c(t)}\sum_{i=1}^p \log\left(s(t)-f_i(\bx,t)\right), \ \bx\in \mathcal{\tilde{D}}(t)
\end{align}
where $\tilde{\mathcal{D}}(t):=\{x \in \mathbb{R}^n\ | f_i(\bx,t)<s(t),\ i=1,\cdots,p \}$ is the perturbed domain. For any initial point $\bx_0 \in \mathbb{R}^n$, $s_0$ can be chosen large enough such that $x_0 \in \mathcal{\tilde{D}}_0$. More precisely, we must have that
\begin{align} \label{eq: slack_initial_value}
s_0=\begin{cases} 0 & \mbox{if} \ \max_{i} \ f_i(\bx_0,0) \leq 0 \\
\max_{i} f_i(\bx_0,0)+\varepsilon & \mbox{if} \ \max_{i} \ f_i(\bx_0,0)> 0
\end{cases}
\end{align}
for some $\varepsilon>0$. Denote by $\tilde{\bz}^{\star}(t)$ as the minimizer of the perturbed barrier function \eqref{eq: perturbed_time_varying_barrier_function}, i.e.
\begin{align} \label{eq: perturbed_inequality_time_varying_minimizer}
\tilde{\bz}^{\star}(t) := & \arg \underset{\bx \in \mathcal{\tilde{D}}(t)}{\min}  \tilde{\Phi}(\bx,t)
\end{align} 
For solving \eqref{eq: perturbed_inequality_time_varying_minimizer}, we can now apply time-varying Newton method  to get the following differential equation
\begin{align} \label{eq: perturbed_time_varying_newton_barrier}
\dot{\tilde{\bz}}(t)=-\nabla_{{\bx}{\bx}}^{-1}\tilde{\Phi}(\tilde{\bz}(t),t) \left( \mathbf{P}\nabla_{{\bx}}\tilde{\Phi}(\tilde{\bz}(t),t)+\nabla_{{\bx} t}\tilde{\Phi}(\tilde{\bz}(t),t) \right),
\end{align}
By increasing the barrier parameter $c(t)$ over time the sub-optimality is decreased as we show in the next lemma.
\begin{lemma}\label{lemma_suboptimality_bound}
	With $\tilde{\bz}^{\star}(t)$ defined as in \eqref{eq: perturbed_inequality_time_varying_minimizer} and $\tilde{\bx}^{\star}(t)$ as in \eqref{eq: perturbed_inequality_constrained_time_varying_problem}, the following inequality holds for all $t \geq 0$
	\begin{align} \label{eq: suboptimality_bound}
	0 \leq f_0(\tilde{\bz}^{\star}(t),t)-f_0(\tilde{\bx}^{\star}(t),t) \leq \dfrac{p}{c(t)}
	\end{align} 
\end{lemma}
\begin{proof}
The proof is provided in Appendix subsection \ref{ap_suboptimality_bound}.
\end{proof}
	The bound  \eqref{eq: suboptimality_bound} quantifies the sub-optimality of $\tilde{\bz}^{\star}(t)$ with respect to the perturbed optimal trajectory $\tilde{\bx}^{\star}(t)$. More specifically, as $c(t) \to \infty$, $\tilde{\bz}^{\star}(t) \to \tilde{\bx}^{\star}(t)$. The next lemma established this convergence under the dynamics \eqref{eq: perturbed_time_varying_newton_barrier}.

\begin{lemma}	\label{lemma: time_varying_barrier_convergence}
	Consider problem \eqref{eq: perturbed_inequality_constrained_time_varying_problem} with the corresponding time-varying barrier function \eqref{eq: perturbed_time_varying_barrier_function}. Let $\tilde{\bz}(t)$ be the solution of the differential equation \eqref{eq: perturbed_time_varying_newton_barrier} with $\mathbf{P} \in \mathbb{S}^{n}_{++}$ satisfying $\mathbf{P}\succeq\sigma{\mathbf{I}}$, and initial condition $\bx_0 \in \mathbb{R}^n$. Finally, let $s_0$ be chosen as in \eqref{eq: slack_initial_value}. Then, the following inequality holds
	\begin{align} \label{eq: time_varying_barrier_convergence}
	\|\tilde{\bz}(t)-\tilde{\bz}^{\star}(t)\|_2 \leq C(\bx_0,c_0,s_0,m)e^{-\sigma t}
	\end{align}
	where $0 \leq C(\bx_0,c_0,s_0,m)<\infty$.
	
	\begin{proof}
		The proof is provided in Appendix subsection \ref{ap: lemma_time_varying_barrier_convergence}. 
	\end{proof}
\end{lemma}
Few comments are in order: First, Lemma \ref{lemma: time_varying_barrier_convergence} shows that the solution $\tilde{\bz}(t)$ of the dynamical system \eqref{eq: perturbed_time_varying_newton_barrier} converges exponentially to the sub-optimal solution $\tilde{\bz}^{\star}(t)$ in \eqref{eq: perturbed_inequality_time_varying_minimizer}. Second, Lemma \ref{lemma_suboptimality_bound} guarantees that $\tilde{\bz}^{\star}(t)$ converges to the optimal perturbed solution $\tilde{\bx}^{\star}(t)$ in \eqref{eq: perturbed_inequality_constrained_time_varying_problem} if $c(t)\to \infty$. Finally, Lemma \ref{lemma: pertubation_suboptimality_bounds} confirms that the perturbed solution $\tilde{\bx}^{\star}(t)$ converges to $\bx^{\star}(t)$ of the original problem \eqref{eq: inequality_constrained_time_varying_problem} if $s(t)\to 0$. Therefore, convergence of the dynamical system \eqref{eq: perturbed_time_varying_newton_barrier} to the desired optimal solution $\bx^{\star}(t)$ is guaranteed if $\lim_{t\to\infty} c(t) =\infty$ and $s(t) \to 0$ as $t\to \infty$. The next theorem summarizes these observations as the main result of this section.


\begin{theorem}\label{thm_main_theorem}
	Consider problem \eqref{eq: inequality_constrained_time_varying_problem} with optimal trajectory $\bx^{\star}(t)$ and the corresponding barrier function \eqref{eq: perturbed_time_varying_barrier_function}. Let $\tilde{\bz}(t)$ be the solution of \eqref{eq: perturbed_time_varying_newton_barrier} with arbitrary initial condition $\bx_0 \in \mathbb{R}^n$. Finally, let $c(t) \to \infty$ and $s(t)=s_0 \exp(-\alpha t)$ for some $\alpha>0$ and $s_0$ chosen according to \eqref{eq: slack_initial_value}. Then, $\tilde{\bz}(t) \to \bx^{\star}(t)$ as $t \to \infty$. 
\end{theorem}
\begin{proof}
The proof follows from inequalities \eqref{eq: pertubation_suboptimality_bounds}, \eqref{eq: suboptimality_bound} and \eqref{eq: time_varying_barrier_convergence}.
\end{proof}
We close this section by two remarks.
\begin{remark}
	The logarithmic barrier coefficient $c(t)$ is required to be positive, monotonic increasing, asymptotically converging to infinity, and be bounded in finite time. A convenient choice could be ${c}(t)=c_0 \exp(\alpha t)$ for $\alpha>0$. Notice that the term $\nabla_{\bx t} \Phi(\tilde{\bz}(t),t)$ in \eqref{eq: perturbed_time_varying_newton_barrier} compensates for continuous-time variation of both $c(t)$ and $s(t)$.
\end{remark}
\begin{remark}
	If the Newton differential equation \eqref{eq: perturbed_time_varying_newton_barrier} starts from a strictly feasible initial point, i.e. $\bx_0 \in \mathcal{D}_0$, then $s(t)$ is chosen to be identically zero. From inequality \eqref{eq: pertubation_suboptimality_bounds}, $\tilde{\bx}(t)=\bx(t)$ for $t \geq 0$. Therefore, according to inequalities \eqref{eq: suboptimality_bound} and \eqref{eq: time_varying_barrier_convergence}, exponential convergence of the tracking trajectory $\tilde{\bz}(t)$ to the optimal trajectory $\bx^{\star}(t)$ is guaranteed if the barrier parameter $c(t)$ grows exponentially.
\end{remark}

\section{Numerical experiments}\label{sec_numerical_examples}
%
\begin{figure}\centering
	\includegraphics[width=\linewidth]{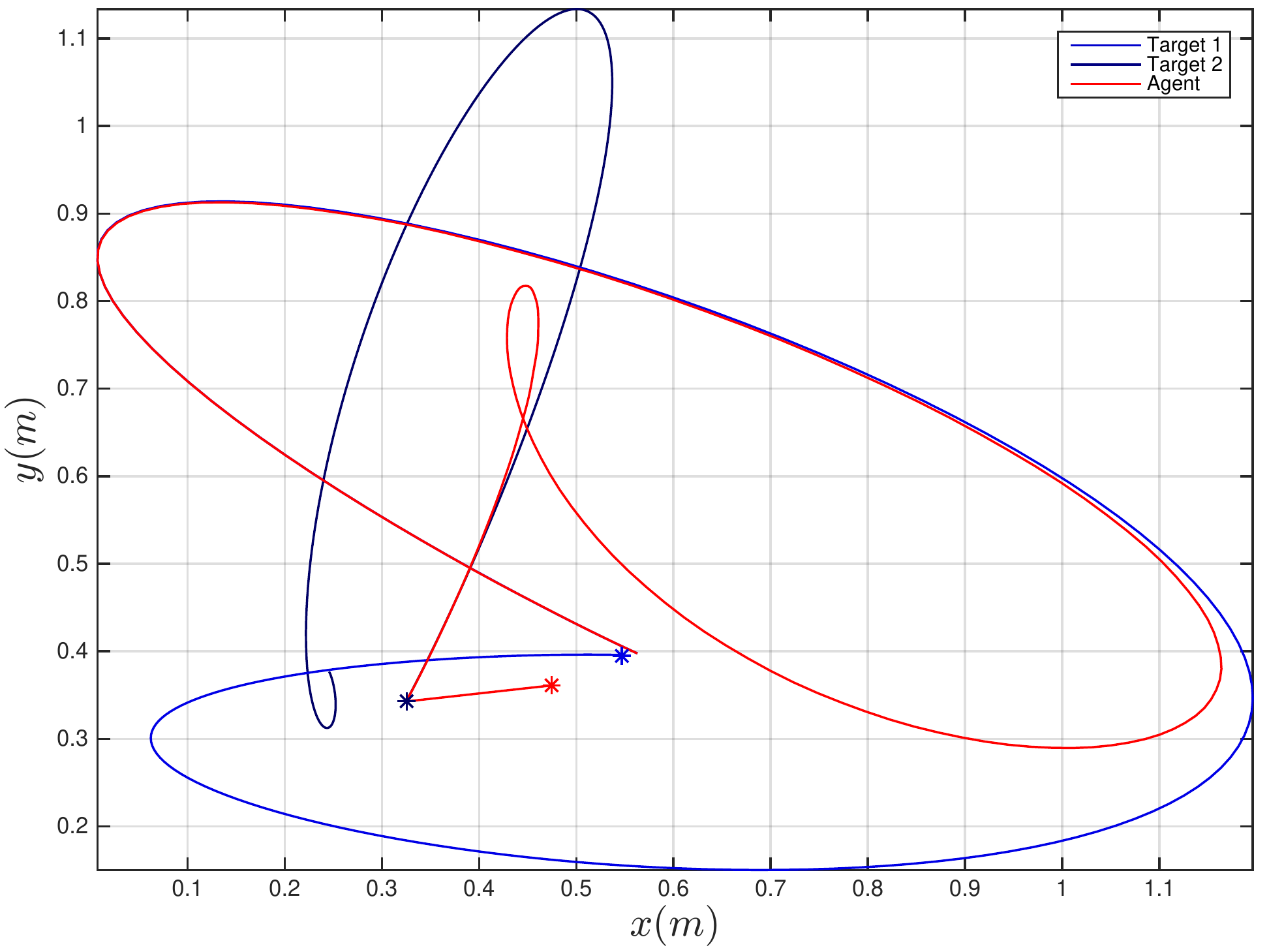} 
	\caption{Trajectory of the agent --in red-- and the targets. The agent starts off with the first target and then switches to the second one. The gain matrix is set to be $\mathbf{P}=10I_2$.}
	\label{fig_trajectory_switched_tracking}\end{figure}
%
%
\begin{figure}\centering
	\includegraphics[width=\linewidth]{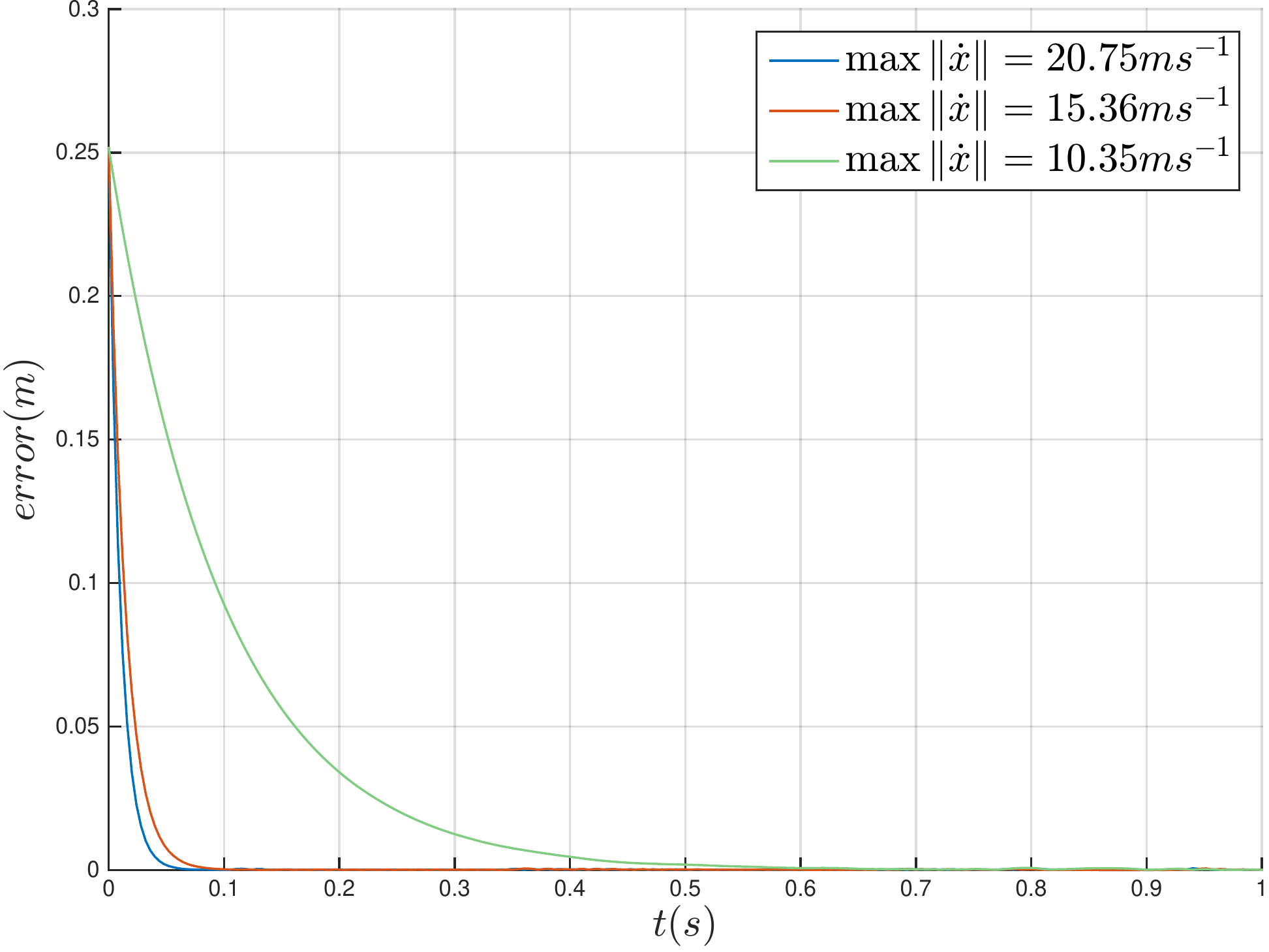} 
	\caption{Tracking error $\| \bx(t) - \bx^*(t)\|$ as a function of time for the unconstrained problem. The optimal solution $\bx^{\star}(t)$ has been computed in discrete times of Euler integration using CVX \cite{grant2008cvx}.}
	\label{fig_double_tracking_error}\end{figure}
%
In this section, we evaluate the performance of the time varying interior point method developed in the previous section by two numerical examples. In Section \ref{sec_unconstrained_minimization} we numerically solve an unconstrained dynamic convex problem and in Section \ref{sec_constrained_optimization} we study a constrained dynamic convex problem. In implementations, we use Euler integration scheme with variable step size to solve the differential equations that generate the solution.

\subsection{Unconstrained optimization}\label{sec_unconstrained_minimization}

Consider an agent charged with the task of tracking two targets sequentially. That is, the agent is required to track the first target on the time interval $[t_0,t_{int}]$ and track the other target on the time interval $[t_{int},t_f]$. If we denote the position of the $i$th target by $\by_i(t)$, the objective function takes the form
\begin{equation}\label{eqn_unconstrained_objective_function}
f_0(\bx, t) = S(t) \| \bx -\by_1(t)\|^2 + (1-S(t))\| \bx-\by_2(t)\|^2,
\end{equation}
where $S(t)$ is a weighting function that determines which target must be tracked. We consider the following differentiable switch
\begin{equation}
S(t) = 1-\frac{1}{1+e^{-\gamma(t-t_{int})}},
\end{equation}  
where $\gamma>0$ controls the speed with which $S(t)$ transitions from one to zero. 

We use a time parametric representation of the trajectories of the targets. Specifically, let $p_j(t)$ be elements of a polynomial basis, the $k$th component of the trajectory of target $i$th is given by
\begin{equation}\label{eqn_sheep_position_unconstrained}
y_{ik} (t) =\sum_{j=0}^{n_i-1} y_{ikj} p_j(t), 
\end{equation}
where $n_i$ is the total number of polynomials that parametrize the path traversed by target $i$ and $y_{ik,j}$ represent the corresponding $n_i$ coefficients. To determine the coefficients $y_{ik,j}$ we draw at random a total of $L$ random points per target $\{\tilde{y}_{\ell}^i\}_{\ell=1}^L$ independently and uniformly in the unit box $[0,1]^2$. Target $i$ is required to pass trough the points $\tilde{y}_{\ell}^i$ at times $\ell t_f/(L+1)$. Paths $y_i(t)$ are then chosen such that the path integral of the acceleration squared is minimized subject to the constraints of each individual path, i.e;
\begin{alignat}{2}\label{eqn_quadratic program_unconstrained}
y_{i}  
= &\arg \min && \int_{0}^T \|\ddot{y}_{i} (t)\|^2 dt,   \nonumber\\
&s.t     && y_i(\ell t_f/(L+1))=\tilde{y}_{\ell}^i \quad \mbox{for all} \quad \ell =0,1\ldots L+1.
\end{alignat}
This problem can be solved by a quadratic program \cite{mellinger2011minimum}. 
In subsequent numerical experiment, we set the number of targets to $m=2$, the time interval to $[0,1]$, and the intermediate switching time $t_{int}=1/2$. We use the standard polynomial basis $p_j(t) = t^j$ in \eqref{eqn_sheep_position_unconstrained} and the degree of the polynomials is set to be $n=30$ for $i=1,\cdots,m$. To generate the target paths, we consider a total of $L=5$ random chosen intermediate points. We further set the parameter controlling the switching speed to $\gamma =20$. For this data, we solve \eqref{eq: time_varying_newton} by Euler integration with variable step size of maximum length $0.01s$. 

The resulting trajectories are illustrated in Figure \ref{fig_trajectory_switched_tracking} when we select the gain matrix  in \eqref{eq: time_varying_newton} to be $\mathbf{P}= 10\mathbf{I}_2$. A qualitative examination of this behavior shows that the agent --in red -- succeeds in tracking the first target up to time $t=0.5 s$ and switching to the second agent after $t=0.5 s$. The objective function considered in \eqref{eqn_unconstrained_objective_function} agrees with Assumption \ref{assumption: strong convexity} and therefore, the hypothesis of Lemma \ref{lem: Time_Varying_Newton} is satisfied. Consequently, exponential  convergence to the optimal solution is guaranteed as it can be observed in Figure \ref{fig_double_tracking_error}. Also, this convergence gets faster as the gain matrix $\mathbf{P}$ is increased.
\subsection{Constrained optimization}\label{sec_constrained_optimization}
%
%
\begin{figure}\centering
	\includegraphics[width=\linewidth]{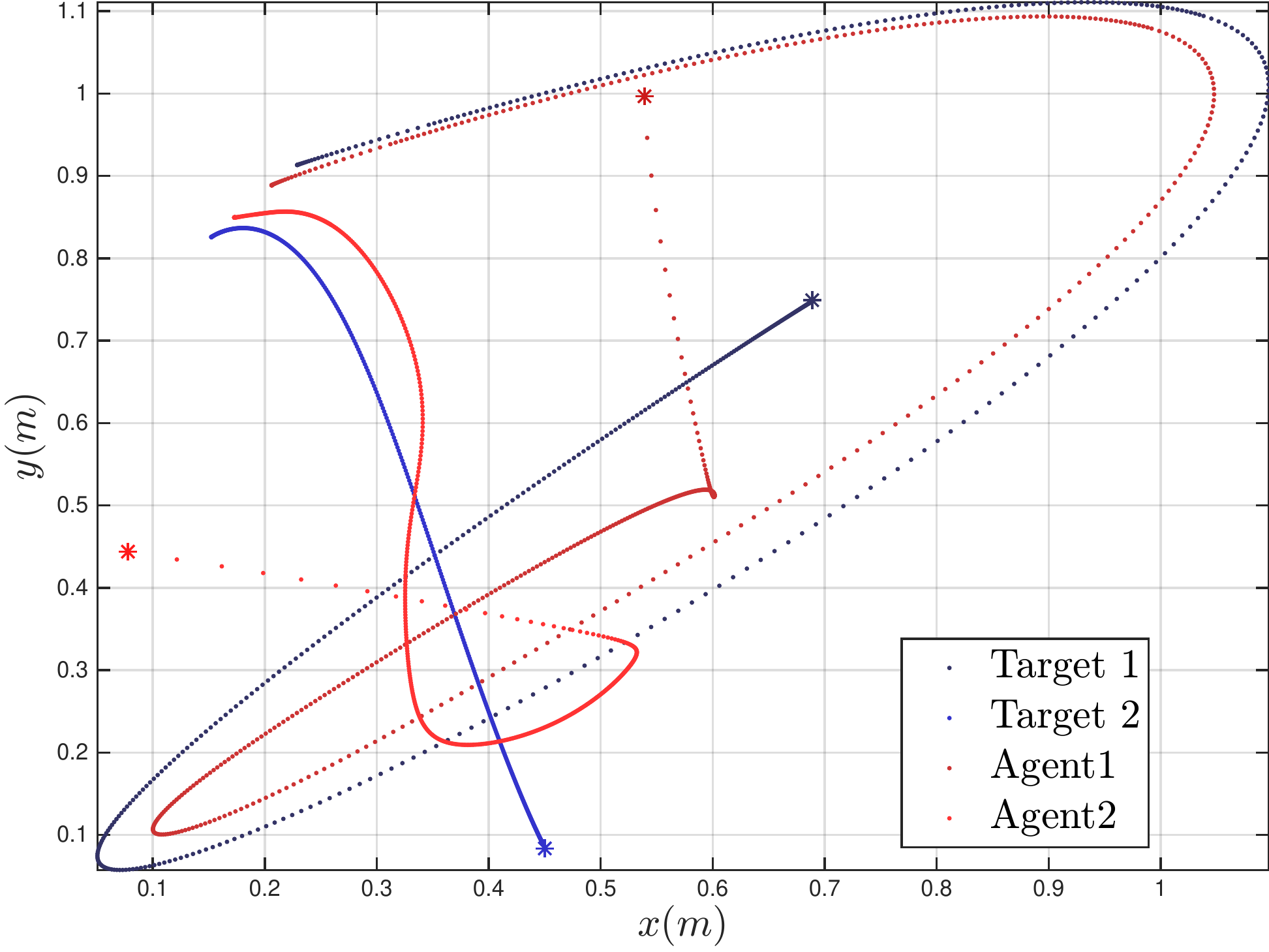} 
	\caption{Trajectory of the agent --in red-- and the targets to track. It can be observed that the agents succeed in following both targets while keeping the distance between them as small as possible. The gain matrix is set to be $\mathbf{P}=50\mathbf{I}_2$.}
	\label{fig_double_tracking_gain50_tol0point05}\end{figure}
\begin{figure}
	\centering
	\begin{subfigure}[b]{\linewidth}
		\includegraphics[width=\textwidth]{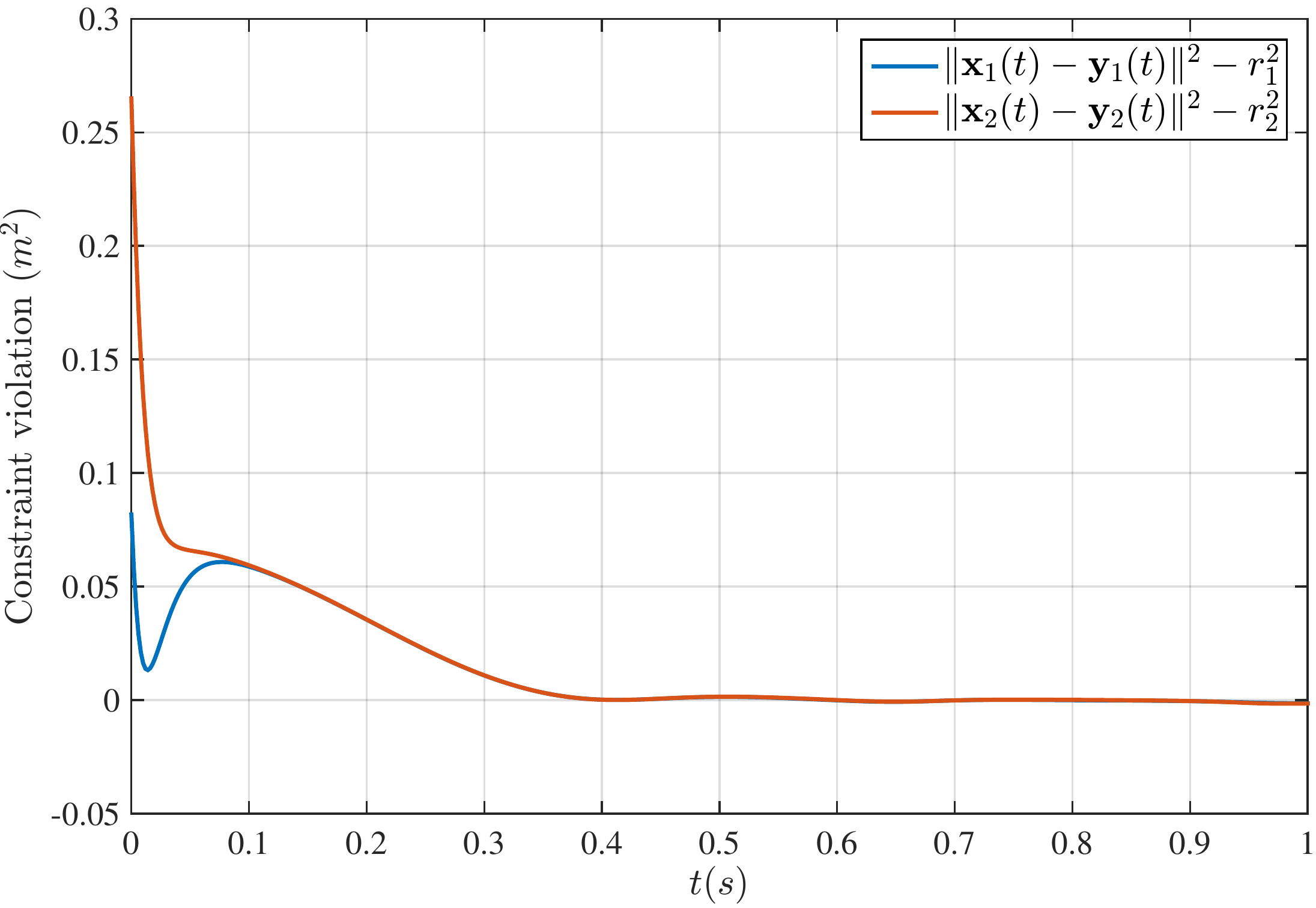}
		\caption{Time evolution of constraint functions $\|\bx_i(t)-\by_i(t)\|_2-r_i^2$ for $i=1,2$.} \label{fig_two_agent_constraint}
	\end{subfigure}\par\vfill \bigskip
	\begin{subfigure}[b]{\linewidth}
		\includegraphics[width=\linewidth]{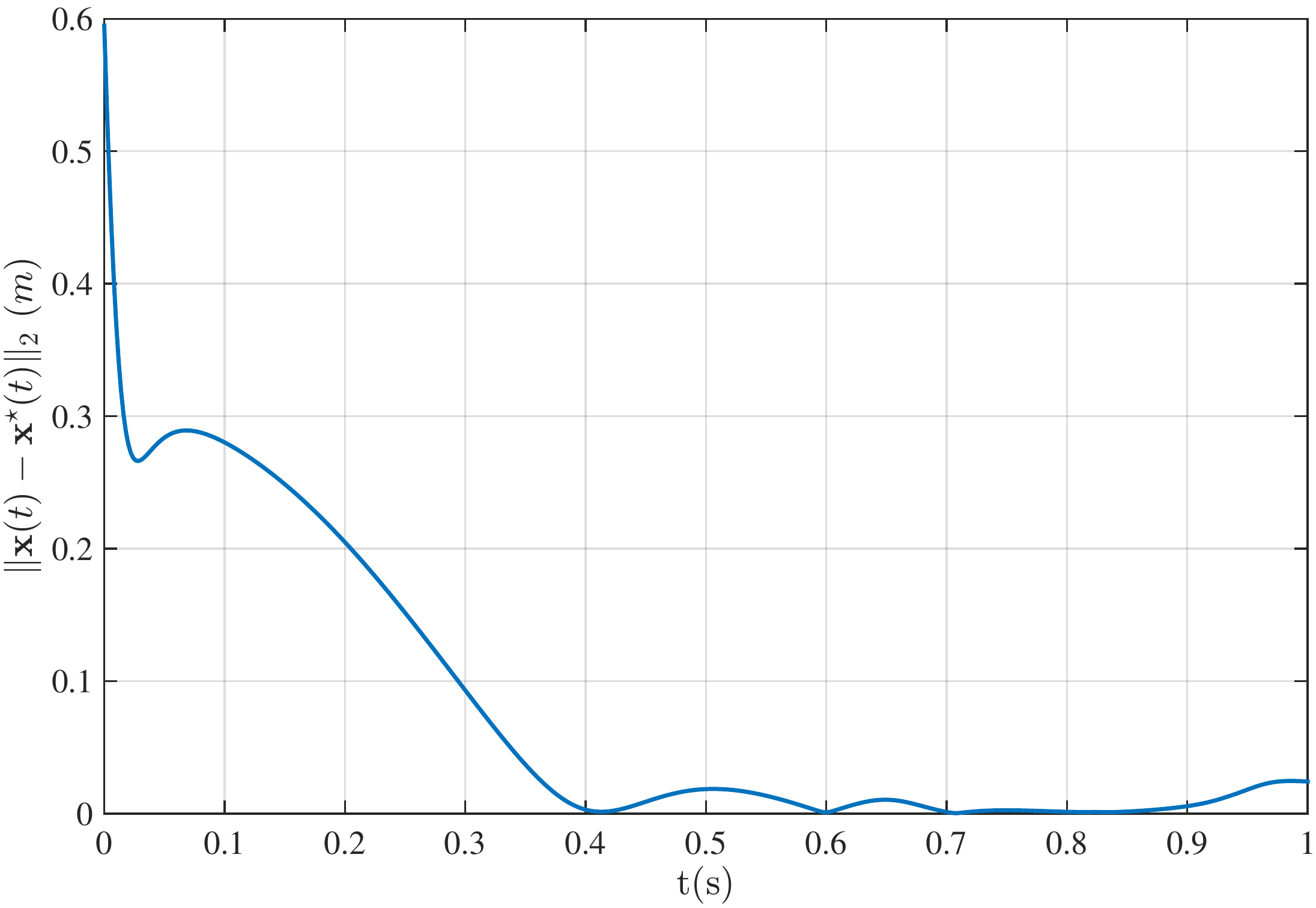}
		\caption{Time evolution of $\left\| [\bx_1(t) - \bx_1^{\star}(x), \bx_2(t)-\bx_2^{\star}(t)]\right\|_2$ where $\bx_1^{\star}(t)$ and $\bx_2^{\star}(t)$ are the optimal trajectories.} \label{fig_two_agent_error}
	\end{subfigure}\bigskip
	\caption{ Evolution of the constraint functions and the tracking error as a function of time. Both the constraint violation and the tracking error converges asymptotically to zero as per Theorem \ref{thm_main_theorem}.} \label{fig_two_agent_optimization}
\end{figure}

We consider two agents charged with the task of staying withing certain distance of two moving targets, while keeping their Euclidean distance as small as possible. Since the position of both agents are optimization variables, the objective function in this problem is not time varying. However, the constraints are. Denote by $\bx_i(t)$ the position vector of agent $i$ and denote by $\by_j(t)$ as the position vector of target $j$. Then, the agents aim to solve the following problem
\begin{equation}\label{eqn_constrained_problem_tracking}
\begin{aligned}
&\min _{\bx_1 \in \mathbb{R}^2, \bx_2 \in \mathbb{R}^2} \| \bx_1 - \bx_2\|^2 \\
&\mbox{s.t} \, \| \bx_i - \by_i(t)\|^2 - r_i^2 \leq 0 \quad \mbox{for} \quad i=1,2.
\end{aligned}
\end{equation}
The trajectories for the targets were computed by the same procedure as in Section \ref{sec_unconstrained_minimization}. The maximum allowable distance to the targets is set to be $r_i = 0.05 m$. The barrier parameter is chosen as $c(t) = e^{\gamma t}$ and the slack parameter as $s(t) = e^{-\alpha t}$ with $\gamma =6$ and $\alpha =10$. For this data, we solve the differential equation \eqref{eq: perturbed_time_varying_newton_barrier} by Euler integration scheme with variable step of maximum length $0.01s$.

In Figure \ref{fig_double_tracking_gain50_tol0point05} the resulting trajectories are depicted. Both agents succeed in following the corresponding target, while keeping their distance small. Figure \ref{fig_two_agent_constraint} illustrates the time evolution of the constraint functions. The value of both constraints converges to zero asymptotically as expected according to Theorem \ref{thm_main_theorem}. It is also expected that the solution of the dynamical system \eqref{eq: perturbed_time_varying_newton_barrier} converges to the optimal solution asymptotically. This convergence is depicted in Figure \ref{fig_two_agent_error}.

\section{Conclusions}\label{se:conclusions}
In this paper, we developed an interior point framework for solving convex optimization problems with time varying objective function and/or time varying constraints. We used barrier penalty functions to relax the constraints and developed a prediction-correction Newton method for solving the corresponding time varying unconstrained problem. Under reasonable assumptions, asymptotic convergence to the optimal solution of the original problem was guaranteed. All time dependences were assumed to be continuous. Numerical examples regarding target tracking applications were considered to illustrate the performance of the developed methods. The numerical results were in accordance with the theoretical findings.
%
\appendix
\section{Appendix}

	\subsection{Proof of Lemma \ref{lem: Time_Varying_Newton}} \label{ap: time_varying_unconstrained_newton}
	
		By Assumption \ref{assumption: strong convexity}(strong convexity), ${\nabla_{\bx\bx} f_0(\bx(t),t)}^{-1}$ is defined for all $t \geq 0$. The time variation of the gradient at $\bx(t)$ can be written as 
		\begin{align} \label{eq: gradient_dynamics}
		\dfrac{d}{dt} \nabla_{\bx} f_0(\bx(t),t)= \nabla_{\bx\bx} f_0(\bx(t),t) \dot{\bx}(t)+ \nabla_{\bx t} f_0(\bx(t),t). 
		\end{align}
		Substituting $\dot{\bx}(t)$ in \eqref{eq: time_varying_newton}, it follows that
		\begin{align}\label{eqn_diff_equation_gradient}
		\dfrac{d}{dt} \nabla_{\bx} f_0(\bx(t),t)=-\mathbf{P} \nabla_{\bx} f_0(\bx(t),t). 
		\end{align}
		This is a first order linear differential equation on $\nabla_{\bx}f_0(\bx(t),t)$. Therefore, the solution of \eqref{eqn_diff_equation_gradient} is 
		\begin{align}
		\nabla_{\bx} f_0(\bx(t),t)=e^{-Pt}\nabla_{\bx} f_0(\bx_0,0).
		\end{align}
		The norm of $\nabla_{\bx} f_0(\bx(t),t)$ can be upper bounded using Cauchy-Schwartz inequality 
		\begin{align}
		\|\nabla_{\bx} f_0(\bx(t),t)\|_2  \leq \|e^{-Pt}\|_2 \|\nabla_{\bx} f_0(\bx_0,0)\|_2.
		\end{align}
		Using the fact that $\sigma \mathbf{I}_n \preceq \mathbf{P}$, we have that $\|e^{-\mathbf{P}t}\|_2=e^{-\sigma t}$, and therefore
		\begin{equation}\label{eqn_lemma1_aux}
		\|\nabla_{\bx} f_0(\bx(t),t)\|_2  \leq e^{-\sigma t} \|\nabla_{\bx} f_0(\bx_0,0)\|_2.
		\end{equation}
		On the other hand, $\|\bx(t) -\bx^*(t)\|_2$ is bounded above by $2\|\nabla_{\bx}f_0(\bx(t),t)  \|_2/m$ due to strong convexity. Hence, it follows that
		\begin{equation} \label{eq: unconstrained_time_varying_convergence_bound}
		\|\bx(t)-\bx^{\star}(t)\|_2  \leq \dfrac{2}{m}\|\nabla_{\bx} f(\bx_0,0)\|_2 e^{-\sigma t}.
		\end{equation}
		The above inequality completes the proof of the lemma. Notice that in fact $C(\bx_0,m) =  2\|\nabla_{\bx}f_0(\bx_0,0)  \|_2/m$ is a bounded nonnegative constant for all initial condition $\bx_0$. 
		
		\subsection{Proof of Lemma \ref{lemma: equality_constrained_time_varyging_newton}} \label{ap: time_varying_equality_newton}

		The Hessian of the Lagrangian \eqref{eq: time_varying_equality_lagrangian} with respect to $\bz=[\bx^T, \ \lambda^T]^T$ is given by
		\begin{align}
		\nabla_{\bz\bz} \mathcal{L}(\bz(t),t)=\begin{bmatrix} \nabla_{\bx\bx} f_0(\bx(t),t) & \mathbf{A}(t)^T \\ \mathbf{A}(t) & 0_{p\times p}
		\end{bmatrix}
		\end{align}
		Strong convexity of $f_0(\bx,t)$ is sufficient for $\nabla_{\bz\bz} \mathcal{L}(\bz,t)$ to be invertible. The rest of the proof follows the same steps as the proof of Lemma \ref{lem: Time_Varying_Newton} by substitutions $\nabla_{\bx} f_0(\bz,t) \leftarrow \nabla_{\bz} \mathcal{L}(\bz,t)$ for the objective function and $\bx \leftarrow \bz$ for the optimization variables.
		
		\subsection{Proof of Lemma \ref{lemma: pertubation_suboptimality_bounds}} \label{ap_perturbation_suboptimality_bound}
		The proof is taken from \cite{boyd2004convex}. 
		Recall that the dual function associated with problem \eqref{eq: inequality_constrained_time_varying_problem} is given by
		\begin{align} \label{eq: time_varying_dual_function}
		g({\lambda}(t),t)= \underset{\bx \in \mathbb{R}^n}{\min} \ f_0(\bx,t)+\sum_{i=1}^{p} \lambda_i(t)f_i(\bx,t),
		\end{align}
		which is concave in $\lambda=[\lambda_1,\cdots,\lambda_p]^T \in \mathbb{R}^{p}_{+}$. At the optimal point $(\bx^{\star}(t),\lambda^{\star}(t))$, it follows by strong duality that		
		\begin{align} \label{eq: timve_varying_strong_duality}
		f_0(\bx^{\star}(t),t)=g(\lambda^{\star}(t),t)
		\end{align}
		On the other hand, for any feasible point $\bx$ of the perturbed problem, i.e. $\bx \in \mathcal{D}_t=\{\bx \in \mathbb{R}^n |\ f_i(\bx,t) \leq s(t) \}$, it must be true that 
		\begin{align} \label{eq: time_varying_weak_duality}
		g(\lambda^{\star}(t),t) &\leq f_0(\bx,t)+\sum_{i=1}^{p} \lambda_i^{\star}(t)f_i(\bx,t) \nonumber \\
		& \leq f_0(\bx,t)+\sum_{i=1}^{p} \lambda_i^{\star}(t)s(t)
		\end{align}
		Substituting \eqref{eq: timve_varying_strong_duality} back in \eqref{eq: time_varying_weak_duality} and setting $\bx=\tilde{\bx}^{\star}(t) \in \mathcal{D}_t$ yields
		\begin{align}
		f_0(\bx^{\star}(t),t) \leq f_0(\tilde{\bx}^{\star}(t),t)+\sum_{i=1}^{p} \lambda_i^{\star}(t)s(t)
		\end{align} 				
		It remains to prove that $f_0(\tilde{\bx}^{\star}(t),t) \leq f_0(\bx^{\star}(t),t)$. This follows from the fact that for $0 \leq s$, the feasible set is enlarged, causing reduction in the optimal value. The proof is complete.
		\subsection{Proof of Lemma \ref{lemma_suboptimality_bound}}\label{ap_suboptimality_bound}
		The optimal trajectory $\tilde{\bz}^{\star}(t)$ solution to the problem \eqref{eq: perturbed_inequality_time_varying_minimizer} is characterized by $\nabla_{\bx} \Phi(\tilde{\bz}^{\star}(t),t)=0$ for all $t \in [0,\infty)$, or equivalently
		\begin{align}	\label{eq: zero_gradient_time_varying_barrier_lagrangian}
		\nabla_{\bx} f_0(\tilde{\bz}^{\star}(t),t)+\dfrac{1}{c(t)}\sum_{i=1}^{p}\dfrac{1}{s(t)-f_i(\tilde{\bz}^{\star}(t),t)}\nabla f_i(\tilde{\bz}^{\star}(t),t)=0,
		\end{align}
		Define the following functions
		\begin{equation}\label{eqn_lambda_t}
		\tilde{\lambda}_i^{\star}(t) = \dfrac{1}{c(t)}\dfrac{1}{s(t)-f_i(\tilde{\bz}^{\star}(t),t)},\ i=1,\cdots,p.
		\end{equation} 
		Notice that $\tilde{\lambda}_i^{\star}(t)>0$ for all $t\geq 0$ and $i \in \{1,\ldots p\}$ since the optimal point is always in the domain, i.e. $f_i(\tilde{\bz}^{\star}(t),t)<s(t), \ i\in \{1,\ldots,p\}$. 
		On the other hand, the dual function of the perturbed problem \eqref{eq: perturbed_inequality_constrained_time_varying_problem} is given by
		\begin{align} \label{eq: perturbed_time_varying_dual_function}
		\tilde{g}(\lambda(t),t) = \underset{\bx \in \mathbb{R}^n}{\min} \ f_0(\bx,t)+\sum_{i=1}^{p} \lambda_i(t)(f_i(\bx,t)-s(t)),
		\end{align}			
		It follows from \eqref{eq: zero_gradient_time_varying_barrier_lagrangian} that the pair $(\tilde{\bz}^{\star}(t),\tilde{\lambda}^{\star}(t))$ satisfies the identity \eqref{eq: perturbed_time_varying_dual_function}. Therefore, we have that
		\begin{align} \label{eq: dual_at_lambda_star}
		\tilde{g}(\tilde{\lambda}^{\star}(t),t) &=f_0(\tilde{\bz}^{\star}(t),t)+\sum_{i=1}^{p} \dfrac{1}{c(t)}\frac{f_i(\tilde{\bz}^{\star}(t),t)-s(t)}{s(t)-f_i(\tilde{\bz}^{\star}(t),t)} \nonumber \\
		&=f_0(\tilde{\bz}^{\star}(t),t)-\dfrac{p}{c(t)}
		\end{align}
		On the other hand, the dual function provides a lower bound for the optimal solution (see e.g. \cite{boyd2004convex}) 
		\begin{align} \label{eq: weak_duality} 
		\tilde{g}(\lambda,t) \leq f_0(\tilde{\bx}^{\star}(t),t),\ \forall \lambda \in \mathbb{R}_{+}^{p},\ \forall t \geq 0.
		\end{align}
		In particular, the above inequality holds for $\tilde{\lambda}^{\star}(t)$. Substituting \eqref{eq: dual_at_lambda_star} in \eqref{eq: weak_duality} results in
		\begin{equation}
		f_0(\tilde{\bz}^{\star}(t),t) - \frac{p}{c(t)} \leq f_0(\tilde{\bx}^{\star}(t),t),
		\end{equation}
		which is the desired bound \eqref{eq: suboptimality_bound}.
		\subsection{Proof of Lemma \ref{lemma: time_varying_barrier_convergence}} \label{ap: lemma_time_varying_barrier_convergence}
				The Hessian of $\tilde{\Phi}(\bx,t)$ can be written as
				\begin{align}
				\nabla_{\bx\bx} \tilde{\Phi}=\nabla_{\bx\bx} f_0+\dfrac{1}{c}\sum_{i=1}^{m} \dfrac{\nabla_{\bx}f_i{\nabla_{\bx}f_i}^T}{(s-f_i)^2}+\dfrac{\nabla_{\bx\bx}f_i}{s-f_i}
				\end{align}
				Since $f_0(\bx,t)$ is strongly convex, and $c(t)$ is strictly positive, it follows that $\nabla_{\bx\bx}\tilde{\Phi}$ is $m$-strongly convex for all $t\in[0,\infty)$. The dynamics of $\nabla_{\bx} \tilde{\Phi}$ at point $(\tilde{\bz}(t),t)$ can be written as
				\begin{align}\label{eqn_gradient_dynamics}
				\dfrac{d}{dt} \nabla_{\bx} \tilde{\Phi}(\tilde{\bz}(t),t) =\nabla_{\bx\bx} \tilde{\Phi}(\tilde{\bz}(t),t) \dot{\tilde{\bz}}(t)+ \nabla_{\bx t} \tilde{\Phi}(\tilde{\bz}(t),t).
				\end{align}
				Substituting the dynamics \eqref{eq: perturbed_time_varying_newton_barrier} into the last result admits
				\begin{equation}
				\dfrac{d}{dt} \nabla_{\bx} \tilde{\Phi}(\tilde{\bz}(t),t) =-\mathbf{P}\nabla_{\bx} \tilde{\Phi}(\tilde{\bz}(t),t). 
				\end{equation}
				This implies in turn that 
				\begin{equation}\label{eqn_theo1_aux}
				\|\nabla_{\bx} \tilde{\Phi}(\tilde{\bz}(t),t)\|_{2} \leq e^{-\sigma t}\|\nabla_{\bx} \tilde{\Phi}(\bx_0,0)\|_2
				\end{equation}
				where 
				\begin{equation}
				\left\|\nabla_{\bx} \tilde{\Phi}(\bx_0,0)\right\|_2=\|\nabla_{\bx} f_0(\bx_0,0)+\dfrac{1}{c_0}\sum_{i=1}^{p} \dfrac{\nabla_{\bx} f_i(\bx_0,0)}{s_0-f_i(\bx_0,0)}\|_2 < \infty.
				\end{equation}
				On the other hand, at each time $t$, $\tilde{\Phi}(\bx,t)$ is $m$-strongly convex. Therefore, it follows that
				\begin{align} \label{eqn_theo1_aux_2}
				\|\tilde{\bz}(t)-\tilde{\bx}^{\star}(t)\|_2 \leq \dfrac{2}{m} \|\nabla_{\bx} \tilde{\Phi}(\tilde{\bz}(t),t)\|_{2}, 
				\end{align}
				Substituting \eqref{eqn_theo1_aux} in \eqref{eqn_theo1_aux_2} gives the desired inequality
				\begin{equation}
				\|\tilde{\bz}(t)-\tilde{\bx}^{\star}(t)\|_2\leq \dfrac{2}{m} \|\nabla_{\bx} \tilde{\Phi}(\bx_0,0)\|_{2} e^{-\sigma t}.
				\end{equation}
				The proof is complete.

\bibliographystyle{ieeetrans}
\bibliography{Refs}

\end{document}